\documentclass[amssymb,twoside,12pt]{article}
\usepackage{latexsym,amsmath,graphicx,mathrsfs,amssymb}
\usepackage{amsfonts}
\usepackage{enumitem}
\newtheorem{theorem}{Theorem}[section]
\newtheorem{lemma}[theorem]{{\bf Lemma}}

\newtheorem{rem}[theorem]{{\bf Remark}}

\newtheorem{definition}{Definition}[section]
\numberwithin{equation}{section}
\newenvironment{proof}{\indent{\em Proof:}}{\quad \hfill
$\Box$\vspace*{2ex}}

\begin{document}
\setcounter{page}{1}
\begin{center}
\vspace{0.3cm} {\large{\bf Mild and classical solutions for fractional evolution differential equation}} \\
\vspace{0.4cm}
 J. Vanterler da C. Sousa $^{1}$ \\
vanterler@ime.unicamp.br \\

\vspace{0.30cm}
Thabet Abdeljawad $^{2}$\\
tabdeljawad@psu.edu.sa\\

\vspace{0.30cm}
 D. S. Oliveira $^{3}$\\
oliveiradaniela@utfpr.edu.br\\
\vspace{0.35cm}

\vspace{0.30cm}
$^{1}$ Department of Applied Mathematics, Imecc-Unicamp, 13083-859, Campinas, SP, Brazil. \\
$^{2}$ Department of Mathematics and General Sciences, Prince Sultan University, P. O. Box 66833, 11586 Riyadh, Saudi Arabia.\\

$^{3}$  Coordination of Civil Engineering, Technological Federal University of Parana, 85053-525, Guarapuava, PR, Brazil.
\end{center}

\def\baselinestretch{1.0}\small\normalsize
\begin{abstract}
Investigating the existence, uniqueness, stability, continuous dependence of data among other properties of solutions of fractional differential equations, has been the object of study by an important range of researchers in the scientific community, especially in fractional calculus. And over the years, these properties have been investigated more vehemently, as they enable more general and new results. In this paper, we investigate the existence and uniqueness of a class of mild and classical solutions of the fractional evolution differential equation in the Banach space $\Omega$. To obtain such results, we use fundamental tools, namely: Banach contraction theorem, Gronwall inequality and the $\beta$-times integrated $\beta$-times integrated $\alpha$-resolvent operator function of an $(\alpha,\beta)$-resolvent operator function.

\end{abstract}
\noindent\textbf{Key words:} Fractional evolution differential equation, existence and uniqueness, mild and classical solutions, Banach contraction theorem, Gronwall inequality. \\
\noindent
\textbf{2010 Mathematics Subject Classification:} 26A33, 34A08, 34A12,	34G20, 47Dxx.
\allowdisplaybreaks
\section{Introduction}

It is common to hear from several researchers that better results were obtained when working with derivatives and fractional integrals \cite{18,12,16,oliveira,samko,ze,ze5,ze4}. In fact, as time goes by the theory of fractional calculus has been increasingly consolidated, with well-founded results and consequently applicable results that allow us to better describe reality. Some papers are suggested for a brief read on some applications \cite{13,Frunzo,Guo24,Heydari,samko}.

Differential equations have long appeared in many branches of engineering, physics, among others, and many concepts and methods have been developed to solve various differential equations, as we know they are fundamental tools for describing physical phenomena such as diffusion processes, erythrocyte sedimentation rate, world population growth, computational models, among others \cite{Almeida1,Almeida,Baba,20,19,ze8,ze9,Srivastava,Zafar,17}. On the other hand, we also have an exponential growth in the interest in the existence, uniqueness, stability and continuous dependence of the data of mild, classical, strong solutions of fractional differential equations of the type: functional, impulsive, evolution, among others \cite{Jia1,Jia,ze3,ze6,ze7,Shu,Zhang34}. Therefore, in recent years considerable attention has been given to investigating various types of fractional differential equations and certainly will continue to be studied for many long years.

In 2010 Zhang et. al. \cite{Zhang} investigated the existence and uniqueness of mild solutions for fractional neutral impulsive functional infinite delay integrodifferential systems with nonlocal initial conditions, through the fixed point theorem with the strongly continuous operator semigroup theory. In the same year, Debbouche \cite{Debbouche} presented a similar work, however, involving the existence and uniqueness of classical solutions for a class of nonlinear fractional evolution systems with nonlocal conditions in Banach space. In 2013, Hern\'{a}ndez et al. \cite{Regan} elaborated a work on some errors about abstract fractional differential equations, i.e., they proposed a different approach to this kind of problem involving the existence of mild solutions for a class of abstract fractional differential equations with nonlocal conditions. In 2014, Dubey and Sharma \cite{Dubey} decided to present a paper on mild and classical solutions for fractional functional differential equations with non-local conditions, highlighting the existence and uniqueness of global mild solutions in Banach space using the Banach fixed point theorem. In the same year, Fan \cite{Fan} investigated the existence and uniqueness of mild and strong solutions for a fractional evolution differential equation introduced via the Riemann-Liouville fractional derivative in Banach space.

On the other hand, we highlight the work that the authors Al-Omari and Al-Saadi did in 2018 \cite{Omari}, investigating the existence and uniqueness of mild and strong solutions for a fractional semilinear evolution equation using the Banach fixed point method and the theory of semi-groups. In this sense, innumerable works on existence, uniqueness, attractivity and other properties of solutions of fractional differential equations, have been published, and consolidates this field. As there is a wide range of interesting and relevant work, we suggest some of these for further reading \cite{Dong,Haibo,14,Guswanto,Katatbeh,Li2,Losada,Zhang2}.

Although, there are many works that contribute to the leverage of the scientific field, the area of fractional differential equations requires more and needs more general mathematical tools that can actually open new and relevant doors. For example, the importance of obtaining an inverse Laplace transform with respect to another function is of paramount importance to construct a new theory of mild, strong, weak, classic solutions for fractional differential equations \cite{Jarad}. In this sense, from the readings of the works reported above, we were motivated to perform a work on the existence and uniqueness of mild solutions of the fractional functional differential evolution equations in the sense of the Hilfer fractional derivative in the Banach space $\Omega$.

Consider fractional functional differential evolution equations given by
\begin{equation}\label{eq.4}
^{H}\mathbb{D}_{t_{0}^{+}}^{\alpha ,\beta}u(t)+\mathcal{A}u(t)=f\left(
t,u(t),u(b_{1}(t)),\cdots ,u(b_{r}(t))\right) ,\,\,t\in J/\{t_{0}\}
\end{equation}
satisfying the condition
\begin{equation}\label{eq.41}
I_{t_{0^{+}}}^{1-\gamma}u(t_{0}^{+})+\sum_{k=1}^{p}C_{k}I_{t_{0^{+}}}^{1-\gamma }u(t_{k})=u_{0}
\end{equation}
where ${}^H\mathbb{D}_{t_{0}^{+}}^{\alpha ,\beta }(\cdot )$ is the Hilfer fractional derivative of order $0<\alpha \leq 1$ and type $0\leq \beta \leq 1$, $I_{t_{0^{+}}}^{1-\gamma }(\cdot )$ is the Riemann-Liouville fractional integral of order $1-\gamma (\gamma =\alpha +\beta (1-\alpha ))$, 
$0\leq \gamma \leq 1$ $f$ and $b_{i}$\, $(i=1,2,\cdots ,r)$ are given functions satisfying some assumptions $u_{0}\in \Omega$, $C_{k}\neq 0$\,$(k=1,2,\cdots ,p)$ and $p,r\in \mathbb{N}$.

Here we perform a rigorous analysis of Eq.(\ref{eq.4}) and our main strategy and results can be
summarized as follows:
\begin{enumerate}
   \item A new class of mild and classic solutions has been introduced as we are free to choose $0<\alpha\leq 1$ e $0\leq \beta\leq 1$;
   
   \item The mild and classic solutions are well defined as noted in remark 1;
   
    \item We investigate the existence and uniqueness of mild solutions for the fractional functional differential evolution equation towards Hilfer using the Banach contraction theorem, according to the Theorem \ref{teo2};
    
    \item Also, we investigated the existence and uniqueness of classical solutions for the fractional functional differential evolution equation in the Hilfer sense using the Gronwall inequality, according to the Theorem \ref{teo31} and Theorem \ref{teo65};
    
    \item From the choice of $0<\alpha\leq1$ and the limits of $\beta\rightarrow 1$ or $\beta\rightarrow 0$ in Eq.(\ref{eq.4}), we have a class of fractional functional differential evolution equation as particular cases. Especially, when we choose $\alpha=1$ and one of the limits of $\beta\rightarrow 1$ or $\beta\rightarrow 0$, we have the integer case. As the solution is directly related to the fractional differential equation investigated, therefore, we also have their respective particular cases, i.e., the results investigated here, are also valid for their respective particular cases, preserving the investigated properties;
\end{enumerate}

This article is organized as follows. In section 2, we present spaces with their respective norms that we will work on, and Riemann-Liouville fractional integral definitions with respect to another function and $\psi$-Hilfer fractional derivative. In this sense, we present the concept of $\beta$-times integrated $\alpha$-resolvent operator function of an $(\alpha,\beta)$-resolvent operator function, Banach contraction theorem and Gronwall inequality, fundamental to achieve the proposed results. In section 3, we investigate our first and main result, namely, the existence and uniqueness of mild solutions for Eq. (\ref{eq.4}) towards the Hilfer fractional derivative using the Banach contraction theorem. To finalize the work, section 4, is intended for the second main result, to investigate the existence and uniqueness of classical solutions for Eq.(\ref{eq.4}) via the Gronwall inequality Lemma. Concluding remarks closing the paper.

\section{Preliminaries}

Let the interval $J'=[0,a]$. The weighted space of continuous functions is given by \cite{ze1,ze2}
\begin{equation*}
C_{1-\gamma}(J',\Omega)= \left\{ \psi \in C(J',\Omega), \, t^{1-\gamma} u(t) \in C(J',\Omega) \right\} 
\end{equation*}
where $0 \leq \gamma \leq 1$, with norm
\begin{equation*}
\begin{array}{rll}
\left\Vert u \right\Vert_{C_{1-\gamma}} & = & \displaystyle \sup_{t \in I} \left\Vert t^{1-\gamma} u(t)\right\Vert.
\end{array}
\end{equation*}

Note that $\Omega$ be a Banach space with norm $\left\Vert \cdot \right\Vert _{C_{1-\gamma }}$ and let $\mathcal{A}:\Omega\rightarrow \Omega$ be a closed densely-defined linear operator. For an operator $\mathcal{A}$, let  $D\left( \mathcal{A},\rho \left( \mathcal{A}\right) \right) $ and $\mathcal{A}^{\ast }$ denotes its domain, resolvent set and adjoint, respectively \cite{Chuang10}.

For a Banach space $\Omega$, $\mathscr{L}(\Omega)$ denotes the set of closed linear operator from $\Omega$ into itself. We shall need the class $G(\widetilde{M},\beta)$ of operators $\mathcal{A}$ satisfying the conditions: There exist constants $\widetilde{M} > 0$ and $\beta \in \mathbb{R}$ such that
\begin{enumerate}
\item $\mathcal{A} \in \mathscr{L}(\Omega)$, $\overline{D(\mathcal{A} )} = \Omega$ and $(\beta,+\infty) \subset \rho(-\mathcal{A})$;

\item $\left\Vert (\mathcal{A} + \xi)^{-k}\right\Vert \leq \widetilde{M} (\xi - \beta)^{-k}
$ for each $\xi > \beta$ and $k=1,2,\cdots$.
\end{enumerate}

{\bf Assumption A:} The adjoint operator $\mathcal{A}^{\ast }$ is densely defined in $\Omega^{\ast }$, i.e., $\overline{D(\mathcal{A}^{\ast })}= \Omega^{\ast }$.

\begin{definition}\cite{Chuang10} Let $\alpha > 0$ and $\beta \geq 0$. A function $\mathbb{S}_{\alpha,\beta} : \mathbb{R}_{+} \to \mathscr{L}(\Omega)$ is called a $\beta$-times integrated $\alpha$-resolvent operator function  of an $(\alpha,\beta)$-resolvent operator function {\rm{(ROF)}} if the following conditions are satisfied:

\begin{tabular}{cl}
{\rm{(A)}} & $\mathbb{S}_{\alpha,\beta}(\cdot)$ is strongly continuous on $\mathbb{R}_{+}$\bigskip and $\mathbb{S}_{\alpha,\beta}(0)=g_{\beta+1}(0)I$;\\\bigskip
{\rm{(B)}} & $\mathbb{S}_{\alpha,\beta}(s) \mathbb{S}_{\alpha,\beta}(t)= \mathbb{S}_{\alpha,\beta}(t) \mathbb{S}_{\alpha,\beta}(s)$ for all $t,s \geq 0$;\\\bigskip
{\rm{(C)}} & the function equation $
\mathbb{S}_{\alpha,\beta}(s) I_t^{\alpha} \mathbb{S}_{\alpha,\beta}(t) - I_s^{\alpha} \mathbb{S}_{\alpha,\beta}(s)\mathbb{S}_{\alpha,\beta}(t) $\\
& $=g_{\beta+1}(s) I_t^{\alpha} \mathbb{S}_{\alpha,\beta}(t) - g_{\beta+1}(t) I_s^{\alpha} \mathbb{S}_{\alpha,\beta}(s)$ for all $t,s \geq 0$.
\end{tabular}
\end{definition}

The generator $\mathcal{A}$ of $\mathbb{S}_{\alpha,\beta}$ is defined by
\begin{equation}
D(\mathcal{A}):= \left\{x \in \Omega: \lim_{t \to 0^{+}} \frac{\mathbb{S}_{\alpha,\beta}(t)\, x - g_{\beta+1}(t)\, x}{g_{\alpha+\beta+1}(t)} \,\, {\rm{exists}}    \right\}
\end{equation}
and 
\begin{equation}
\mathcal{A}\,x := \lim_{t \to 0^{+}} \frac{\mathbb{S}_{\alpha,\beta}(t)\, x - g_{\beta+1}(t)\, x}{g_{\alpha+\beta+1}(t)}\, , \quad x \in D(\mathcal{A}),
\end{equation}
where $g_{\alpha+\beta+1}(t):= \dfrac{t^{\alpha+\beta}}{\Gamma(\alpha+\beta)}$ ($\alpha+\beta>0$).

It is known that for $\mathcal{A}\in G(\widetilde{M},\beta )$ there exists exactly one $(\alpha ,\beta )$-resolvent operator function $\mathbb{S}_{\alpha ,\beta }(t):\Omega\rightarrow \Omega$ for $t\geq 0$ such that $-\mathcal{A}$ is the infinitesimal operator and 
\begin{equation*}
\left\Vert \mathbb{S}_{\alpha ,\beta }(t)\right\Vert\leq  \widetilde{M}\,e^{\beta t},\text{ for }t\geq 0.
\end{equation*}

Throughout the paper we shall assume the conditions 1, 2 and assumption {\bf A}. Moreover, we shall use the notation $0\leq t_{0}<t_{1}<\cdots <t_{p}\leq t_{0}+a$, $a>0$; $J:=[t_{0},t_{0}+a]$; 
\begin{equation*}
M:=\sup_{t\in \lbrack 0,a]}\left\Vert \mathbb{S}_{\alpha ,\beta }(t)\right\Vert
\end{equation*}
and $X:=C_{1-\gamma }(J,\Omega)$.

Also, we shall assume that there exists the operator $B$ with $(B)=\Omega$ given by the formula 
\begin{equation*}
B:=\left( \mathrm{I}+\sum_{k=1}^{p}C_{k}I_{0^{+}}^{1-\gamma }\mathbb{S}_{\alpha ,\beta }(t_{k}-t_{0})\right) ^{-1}
\end{equation*}
where $\mathrm{I}$ is the identity operator on $\Omega$.

Let $\left( a,b\right) $ $\left( -\infty \leq a<b\leq \infty \right) $ be a finite interval {\rm{(or infinite)}} of the real line $\mathbb{R}$ and let $\alpha >0$. Also let $\psi \left( x\right) $ be an increasing and positive monotone function on $\left( a,b\right] ,$ having a continuous derivative $\psi ^{\prime }\left( x\right)$ {\rm{(we denote first derivative as $\dfrac{d}{dx}\psi(x)=\psi'(x)$)}} on $\left( a,b\right) $. The left-sided fractional integral of a function $f$ with respect to a function $\psi $ on $ \left[ a,b\right] $ is defined by \cite{ze,ze5}
\begin{equation}\label{eq7}
I_{a+}^{\alpha ;\psi }f\left( x\right) =\frac{1}{\Gamma \left( \alpha \right) }\int_{a}^{x}\psi ^{\prime }\left( s\right) \left( \psi \left( x\right) -\psi \left( s\right) \right) ^{\alpha -1}f\left( s\right) ds.
\end{equation}

On the other hand, let $n-1<\alpha <n$ with $n\in \mathbb{N},$ let $I'=\left[a,b\right] $ be an interval such that $-\infty \leq a<b\leq \infty $ and let $f,\psi \in C^{n}\left[ a,b\right] $ be two functions such that $\psi $ is increasing and $\psi ^{\prime }\left( x\right) \neq 0,$ for all $x\in I'$. The left-sided $\psi -$Hilfer fractional derivative $^{H}\mathbb{D}_{a+}^{\alpha ,\beta ;\psi }\left( \cdot \right) $ of a function $f$ of order $\alpha $ and type $0\leq \beta \leq 1,$ is defined by \cite{ze,ze5}
\begin{equation}\label{eq8}
^{H}\mathbb{D}_{a+}^{\alpha ,\beta ;\psi }f\left( x\right) =I_{a+}^{\beta \left( n-\alpha \right) ;\psi }\left( \frac{1}{\psi ^{\prime }\left( x\right) }\frac{d}{dx}\right) ^{n}I_{a+}^{\left( 1-\beta \right) \left( n-\alpha \right) ;\psi }f\left( x\right),
\end{equation}
where $I^{\alpha}_{a+}(\cdot)$ is $\psi$-Riemann-Liouville fractional integral. 

To study our problem Eq.(\ref{eq.4}), first we shall need the following linear problem, 
\begin{equation}\label{eq.5}
\left\{ 
\begin{array}{rcl}
{}^H\mathbb{D}_{t_{0}^{+}}^{\alpha ,\beta }u(t)+\mathcal{A}u(t) & 
= & g(t),\,\,t\in J/\{t_{0}^{+}\} \\ 
I_{t_{0^{+}}}^{1-\gamma }u(t_{0}^{+}) & = & x%
\end{array}%
\right. 
\end{equation}%
and the following definition.

We will introduce two fundamental results for the investigation of the main results of this paper, namely the Banach contraction theorem and the Gronwall Inequality.

\begin{theorem}{\rm \cite{ze1,ze2}} \label{ban}\textsc{(Banach contraction principle).} Let $\left( X,d\right) $ be a generalized complete metric space. Assume that $\Omega :X\rightarrow X$ is a strictly contractive operator with the Lipschitz constant $L<1$. If there exists a nonnegative integer $k$ such that $d\left( \Omega ^{k+1},\Omega ^{k}\right) <\infty $ for some $x\in X$, then the following are true:
\begin{enumerate}
\item The sequence $\left\{ \Omega ^{k}x\right\} $ converges to a point $x^{\ast }$ of $\Omega$;

\item $x^{\ast }$ is the unique fixed point of $\Omega $ in $\Omega ^{\ast }=\left\{ y\in X/d\left( \Omega ^{k}x,y\right) <\infty \right\}$;

\item If $y\in X^{\ast }$, then $d\left( y,x^{\ast }\right) \leq \dfrac{1}{1-L}d\left( \Omega y,y\right)$.
\end{enumerate}
\end{theorem}

\begin{theorem}{\rm \cite{17,ze10}} \label{gronwalltheorem}\textsc{(Gronwall theorem).} Let $u,v$ be two integrable functions and $g$ a continuous function, it domain $[a,b]$. Let $ \psi \in C^1(\Omega,\mathbb{R})$  an increasing function such that $\psi^{\prime }(t) \neq 0$, $\forall t \in \Omega$. Assume that:  
\begin{enumerate}
\item $u$ and $v$ are nonnegative;
\item $y$ is nonnegative and nondecreasing.
\end{enumerate}

If 
\begin{equation*}
u(t)\leq v(t)+g(t)\int_{a}^{b}\mathbf{Q}_{\psi }^{\mu }\left( t,s\right) u(s)\,{%
\mbox{d}}s,
\end{equation*}%
then 
\begin{equation*}
u(t)\leq v(t)+\int_{a}^{b}\sum_{k=1}^{\infty }\frac{[g(t)\xi (\mu
)]^{k}}{\Gamma (\mu k)}\mathbf{Q}_{\psi }^{k\mu }\left( t,s\right) v(s)\,{%
\mbox{d}}s,
\end{equation*}%
$\forall t\in \Omega$ and $\mathbf{Q}_{\psi }^{k\mu }\left( t,s\right) :=\psi ^{\prime }\left( s\right) \left( \psi \left( t\right) -\psi \left( s\right) \right) ^{k\mu -1}.$
\end{theorem}

\begin{lemma}{\rm \cite{ze10}} \label{gronlema} \textsc{(Gronwall lemma)} Under the hypotheses of {\rm Theorem \ref{gronwalltheorem}}, let $v$ be a non-decreasing function on $\Omega$. Then, we have 
\begin{equation*}
u(t)\leq v(t)\mathbb{E}_{\mu }\left( g(t)\Gamma (\mu )[(\psi (t)-\psi
(a))^{\mu }]\right) 
\end{equation*}%
$t\in \Omega$, where $\mathbb{E}_{\mu }(\cdot )$ is the Mittag-Leffler function with one parameter.
\end{lemma}

\begin{definition} 
A function $u: J \to \Omega$ is said to be a classical solution to the problem {\rm Eq.(\ref{eq.4})}
\begin{enumerate}
    \item $u$ is continuous on $J$ and continuously differentiable on $J/\{t_0^{+}\}$;
    
    \item $ {}^H \mathbb{D}^{\alpha,\beta}_{t_0^{+}} u(t) + \mathcal{A}u(t) = g(t), \,{\rm for}\, t \in J/\{t_0^{+}\}$;\
    
    \item $I_{t_{0^{+}}}^{1-\gamma} u(t_0^{+}) = x $.
\end{enumerate}
\end{definition}

To study problem Eq.(\ref{eq.4}), we shall need the following theorem.

\begin{theorem}\label{teo1} 
Let $y: J \to \Omega$ be Lipschitz continuous on $J$ and $x \in D(\mathcal{A})$. Then, the Cauchy problem {\rm Eq.(\ref{eq.5})} has exactly one classical solution, denoted by $u$, given by the formula
\begin{equation*}
u(t) = \mathbb{S}_{\alpha,\beta}(t-t_0) x + \int_{t_0^{+}}^t K_{\alpha} (t-s) y(s) \, {\rm{d}}s, \quad t \in J.
\end{equation*}
\end{theorem}

On the other hand, a function $u \in \Lambda$ satisfying the integral equation
\begin{eqnarray}\label{mild}
u(t) &=& \mathbb{S}_{\alpha,\beta}(t-t_0) B u_0 + \int_{t_0^{+}}^{t} K_{\alpha}(t - s) f(s,u(s),u(b_1(s)),\cdots,u(b_r(s))) ds \\
&-& \mathbb{S}_{\alpha,\beta}(t-t_0) B \sum_{k=1}^p c_k I_{t_0^{+}}^{1 -\gamma}  \int_{t_0^{+}}^{t_k} K_{\alpha}(t_k - s) f(s,u(s),u(b_1(s)),\cdots,u(b_r(s))) ds \nonumber
\end{eqnarray}
is said to be a mild solution of the fractional functional differential nonlocal evolution Cauchy problem Eq.(\ref{eq.4}) satisfying the condition Eq.(\ref{eq.41}), where $ K_{\alpha }\left( t\right) =t^{\alpha -1}G_{\alpha }\left( t\right)$, \\$G_{\alpha }\left( t\right) =\displaystyle\int_{0}^{\infty }\alpha \theta M_{\alpha }\left( \theta \right) \mathcal{A}\left( t^{\alpha }\theta \right) d\theta $, $\mathbb{S}_{\alpha ,\beta }\left( t\right) =I_{\theta }^{\beta \left( 1-\alpha \right) }K_{\alpha}\left( t\right)$, $0< \alpha\leq 1$ and $0\leq \beta\leq 1$. For more details see \cite{ze6,ze7} and references therein.

\begin{rem}\label{rem1}
A mild solution of the nonlocal fractional Cauchy problem Eq.(\ref{eq.4}) satisfying the condition Eq.(\ref{eq.41}). Indeed, observe that by Eq.(\ref{mild})
\begin{eqnarray}\label{A1}
&& u(t_0) = \mathbb{S}_{\alpha,\beta}(0) B u_0 - \mathbb{S}_{\alpha,\beta}(0) B \sum_{k=1}^p c_k I_{0^{+}}^{1-\gamma} \times  \nonumber \\
&&\times\int_{t_0}^t K_{\alpha}(t_k-s) 
f(s,u(s),u(b_1(s)),\cdots,u(b_r(s)))ds
\end{eqnarray}
and
\begin{eqnarray}\label{A2}
u(t_{i}) &=&\mathbb{S}_{\alpha ,\beta }(t_{i}-t_{0})Bu_{0}+\int_{t_{0}^{+}}^{t_{i}}K_{\alpha
}(t_{i}-s)f(s,u(s),u(b_{1}(s)),\cdots ,u(b_{r}(s)))ds\nonumber \\
&&-\mathbb{S}_{\alpha ,\beta }(t_{i}-t_{0})B\sum_{k=1}^{p}C_{k}I_{t_{0}^{+}}^{1-\gamma
}\times   \\
&&\times\int_{t_{0}^{+}}^{t_{k}}K_{\alpha }(t_{k}-s)f(s,u(s),u(b_{1}(s)),\cdots
,u(b_{r}(s)))\,{ds.}  \nonumber
\end{eqnarray}

From {\rm Eq.(\ref{A1})} and {\rm Eq.(\ref{A2})} and the definition of operator $B$ we obtain the formula
\begin{eqnarray}
&&I_{t_{0}^{+}}^{1-\gamma
}u(t_{0})+\sum_{i=1}^{p}c_{i}I_{t_{0}^{+}}^{1-\gamma }u(t_{i})  
\notag\\ &=&\left( I_{t_{0}^{+}}^{1-\gamma }\mathbb{S}_{\alpha ,\beta
}(0)+\sum_{i=1}^{p}c_{i}I_{t_{0}^{+}}^{1-\gamma }\mathbb{S}_{\alpha ,\beta
}(t_{i}=t_{0})\right) Bu_{0}  \nonumber \\
&&-\left( I_{t_{0}^{+}}^{1-\gamma }\mathbb{S}_{\alpha ,\beta
}(0)+\sum_{i=1}^{p}c_{i}I_{t_{0}^{+}}^{1-\gamma }\mathbb{S}_{\alpha ,\beta
}(t_{i}=t_{0})\right) B\times   \nonumber \\
&&\times \sum_{k=1}^{p}C_{k}I_{t_{0}^{+}}^{1-\gamma
}\int_{t_{0}^{+}}^{t_{k}}K_{\alpha }(t_{k}-s)f(s,u(s),u(b_{1}(s)),\cdots
,u(b_{r}(s)))ds  \nonumber \\
&&+\sum_{i=1}^{p}c_{i}I_{t_{0}^{+}}^{1-\gamma
}\int_{t_{0}^{+}}^{t_{i}}K_{\alpha }(t_{i}-s)f(s,u(s),u(b_{1}(s)),\cdots
,u(b_{r}(s)))\,{ds.}\nonumber \\
&&=u_{0}.
\end{eqnarray}%
\end{rem}

Thus we have that the mild solution given by Eq.(\ref{eq.4}) satisfies the condition given by Eq.(\ref{eq.41}), so it is well defined.
\section{Existence and Uniqueness of mild solution}

In this section, we will investigate the existence and uniqueness of mild solutions for the fractional functional differential evolution equations introduced by the Hilfer fractional derivative and using the Banach contraction principle.
\begin{theorem}\label{teo2} If 
\begin{enumerate}
    \item $f: J \times \Omega^{r+1} \to \Omega$ is continuous with respect to the first variable on $J$, $b_i: J \to \mathbb{N}\,\,(i=1,2,\cdots,r)$ are continuous on $J$ and there is $L > 0$ such that 
\begin{equation}
\left\Vert f\left( s,z_{0},z_{1},\cdots ,z_{r}\right) -f\left( s,\overline{z}%
_{0},\overline{z}_{1},\cdots ,\overline{z}_{r}\right) \right\Vert
\leq L\sum_{i=0}^{r}\left\Vert z_{i}-\overline{z}%
_{i}\right\Vert _{C_{1-\gamma }}
\end{equation}
for $s \in J$, $z_i,\overline{z}_i \in \Omega$, $i=0,1,\cdots,r$;

    \item $(r+1)MLa^{2} \left(1 + M \left\Vert B\right\Vert \widetilde{C} \sum_{k=1}^p |c_k| \right) < 1$;

    \item $u_0 \in \Omega$.
\end{enumerate}
Then, the problem {\rm Eq.(\ref{eq.4})} has a unique mild solution.
\end{theorem}

\begin{proof}
Consider the following operator
\begin{eqnarray}
(F\omega )(t) &:&=\mathbb{S}_{\alpha ,\beta }(t-t_{0})Bu_{0}-\mathbb{S}%
_{\alpha ,\beta }(t-t_{0})B\times   \nonumber \\
&&\times \sum_{k=1}^{p}C_{k}I_{t_{0}^{+}}^{1-\gamma
}\int_{t_{0}^{+}}^{t_{k}}K_{\alpha }(t_{k}-s)f(s,\omega (s),\omega
(b_{1}(s)),\cdots ,\omega (b_{r}(s)))ds  \nonumber \\
&&+\int_{t_{0}^{+}}^{t}K_{\alpha }(t-s)f(s,\omega (s),\omega
(b_{1}(s)),\cdots ,\omega (b_{r}(s)))ds
\end{eqnarray}%
with $t\in J$ and $\omega \in \Lambda$, being $\Lambda$ a Banach space.

Now, we go prove that in fact $F$ is a contraction, and so use Banach contraction theorem. So, we have
\begin{eqnarray*}
&&\left\Vert (F\omega )(t)-(F\widetilde{\omega })(t)\right\Vert   \nonumber \\
&=&-P_{\alpha ,\beta }(t-t_{0})B\sum_{k=1}^{p}C_{k}I_{t_{0}^{+}}^{1-\gamma
}\int_{t_{0}}^{t_{k}}\left[ 
\begin{array}{c}
f\left( s,\omega \left( s\right) ,\omega \left( b_{1}\left( s\right) \right)
,\cdots ,\omega \left( b_{r}\left( s\right) \right) \right)  \\ 
-f(s,\widetilde{\omega }(s),\widetilde{\omega }(b_{1}(s)),\cdots ,\overline{%
\omega }(b_{r}(s)))%
\end{array}%
\right] ds  \nonumber \\
&&+\int_{t_{0}}^{t}K_{\alpha }(t-s)\left[ 
\begin{array}{c}
f(s,\omega (s),\omega (b_{1}(s)),\cdots ,\omega (b_{r}(s))) \\ 
-f(s,\widetilde{\omega }(s),\widetilde{\omega }(b_{1}(s)),\cdots ,\overline{%
\omega }(b_{r}(s)))%
\end{array}%
\right] ds  \nonumber \\
&\leq &\sum_{k=1}^{p}|C_{k}|\left\Vert \mathbb{S}_{\alpha ,\beta
}(t-t_{0})\right\Vert \left\Vert B\right\Vert \left\Vert I_{t_{0}}^{1-\gamma }\right\Vert \times   \nonumber \\
&&\times \int_{t_{0}^{+}}^{t_{k}}\left\Vert K_{\alpha }(t_{k}-s)\right\Vert \left\Vert f(s,\widetilde{\omega }(s),\widetilde{\omega } (b_{1}(s)),\cdots ,\widetilde{\omega }(b_{r}(s)))\right\Vert ds  \nonumber \\
&&+\int_{t_{0}}^{t}\left\Vert K_{\alpha }(t-s)\right\Vert\left\Vert 
\begin{array}{c}
f(s,\omega (s),\omega (b_{1}(s)),\cdots ,\omega (b_{r}(s))) \\ 
-f(s,\widetilde{\omega }(s),\widetilde{\omega }(b_{1}(s)),\cdots ,\overline{\omega }(b_{r}(s)))
\end{array}%
\right\Vert ds  \nonumber \\
&\leq &\sum_{k=1}^{p}|C_{k}|M\left\Vert B\right\Vert 
\widetilde{C}\int_{t_{0}^{+}}^{t_{k}}M\,L\sum_{i=0}^{r}\left\Vert \omega -\widetilde{\omega }\right\Vert _{C_{1-\gamma }}ds+\int_{t_{0}^{+}}^{t}M%
\,L\sum_{i=0}^{r}\left\Vert \omega -\widetilde{\omega }\right\Vert
_{C_{1-\gamma }}ds  \nonumber \\
&=&\widetilde{C}M^{2}Lr\left\Vert \omega -\widetilde{\omega }\right\Vert
_{C_{1-\gamma }}\sum_{k=1}^{p}|C_{k}|\left\Vert B\right\Vert (t_{k}-t_{0})+L\,M\,r\left\Vert \omega -\widetilde{\omega }\right\Vert
_{C_{1-\gamma }}(t-t_{0})  \nonumber \\
&\leq &M^{2}aLr\left\Vert \omega -\widetilde{\omega }\right\Vert
_{C_{1-\gamma }}\left\Vert B\right\Vert \widetilde{C}%
\sum_{k=1}^{p}|C_{k}|+LMra\left\Vert \omega -\widetilde{\omega }\right\Vert
_{C_{1-\gamma }}  \nonumber \\
&=&LMra\left( 1+M\left\Vert B\right\Vert\widetilde{C}%
\sum_{k=1}^{p}|C_{k}|\right) \left\Vert \omega -\widetilde{\omega }%
\right\Vert _{C_{1-\gamma }}  \nonumber \\
&\leq &(1+r)MLa^{2}\left( 1+M\left\Vert B\right\Vert \widetilde{C%
}\sum_{k=1}^{p}|C_{k}|\right) \left\Vert \omega -\widetilde{\omega }%
\right\Vert _{C_{1-\gamma }}.
\end{eqnarray*}

Then, we have
\begin{equation}
\left\Vert F\omega -F\widetilde{\omega }\right\Vert _{C_{1-\gamma }}\leq  \widetilde{q}\left\Vert \omega -\widetilde{\omega }\right\Vert _{C_{1-\gamma }}  \label{star}
\end{equation}
where
\begin{equation*}
\widetilde{q}:=(1+r)MLa^{2}\left( 1+M\left\Vert B\right\Vert _{C_{1-\gamma }}%
\widetilde{C}\sum_{k=1}^{p}|C_{k}|\right) 
\end{equation*}
with $0 < \widetilde{q} < 1$.

Thus, we have that inequality (\ref{star}) satisfies the conditions of the Banach contraction theorem (see Theorem \ref{ban}). Therefore, we can ensure that there is a single $F$ operator fixed point in the space  $\Lambda$ and this point is the mild solution of the problem Eq.(\ref{eq.4}) satisfying the condition Eq.(\ref{eq.41}).
\end{proof}

\section{Existence and uniqueness of classical solutions}

Before beginning to investigate the results of this section, let us present the definition of the classical solution to the problem Eq.(\ref{eq.4}) satisfying the condition Eq.(\ref{eq.41}).

\begin{definition}
A function $u: J \to \Omega$ is said to be a classical solution of the fractional functional differential nonlocal evolution Cauchy problem {\rm Eq.(\ref{eq.4})} if
\begin{enumerate}
    \item $u$ is continuous on $J$ and continuously differentiable on $J/\{t_0^{+}\}$;
    
    \item $ {}^H \mathbb{D}^{\alpha,\beta}_{t_0^{+}} u(t) + \mathcal{A}u(t) = f(t,u(t),u(b_1(t)),\cdots, u(b_r(t))), \,{\rm for}\, t \in J/\{t_0^{+}\}$;
    
    \item $\displaystyle I_{t_{0^{+}}}^{1-\gamma} u(t_0) + \sum_{k=1}^p c_k   I_{t_{0^{+}}}^{1-\gamma} u(t_k)   = u_0 $.
\end{enumerate}
\end{definition}

The first result of this section, is a direct consequence of the $u$ solution being said classical, that is, let's investigate that if $u$ is classical from the problem Eq.(\ref{eq.4}) satisfying the condition Eq.(\ref{eq.41}), so it's mild.

\begin{theorem}\label{teo31}
Assume that $f: J \times \Omega^{r+1} \to \Omega$ is Lipschitz continuous on $J \times \Omega^{r+1}$. If $u$ is a classical solution to the problem {\rm Eq.(\ref{eq.4})} then $u$ is a mild solution of this problem.
\end{theorem}

\begin{proof}
Since $u$ is a classical solution to the problem Eq.(\ref{eq.4}), $u \in \Lambda$ and $u$ satisfies the integral equation, 
\begin{equation}\label{eq3}
u(t) = \mathbb{S}_{\alpha,\beta}(t-t_0) I_{t_{0^{+}}}^{1-\gamma} u(t_0) + \int_{t_0^{+}}^{t} K_{\alpha}(t-s) f(s,u(s),u(b_1(s)),\cdots, u(b_r(s))) ds
\end{equation}
with $t \in J$.

Note that through Eq.(\ref{eq3}) and the condition 3 from the condition of Eq.(\ref{eq.4}), is given by
\begin{eqnarray}
u_{0} &=&I_{t_{0^{+}}}^{1-\gamma
}u(t_{0})+\sum_{k=1}^{p}C_{k}I_{t_{0^{+}}}^{1-\gamma }\mathbb{S}_{\alpha
,\beta }\left( t_{k}-t_{0}\right) \times    \\
&&\times \left( I_{t_{0^{+}}}^{1-\gamma
}u(t_{0})\int_{t_{0}^{+}}^{t_{k}}K_{\alpha
}(t_{k}-s)f(s,u(s),u(b_{1}(s)),\cdots ,u(b_{r}(s)))ds\right) \nonumber
\end{eqnarray}
or it can also be written as
\begin{eqnarray}\label{eq181}
u_{0} &=&\left( {I}+\sum_{k=1}^{p}C_{k}I_{t_{0^{+}}}^{1-\gamma }\mathbb{S}%
_{\alpha ,\beta }(t_{k}-t_{0})\right) I_{t_{0^{+}}}^{1-\gamma }u(t_{0}) 
\\
&&+\sum_{k=1}^{p}C_{k}I_{t_{0^{+}}}^{1-\gamma
}\int_{t_{0}^{+}}^{t_{k}}K_{\alpha }(t_{k}-s)f(s,u(s),u(b_{1}(s)),\cdots
,u(b_{r}(s)))ds.\nonumber
\end{eqnarray}

Applying the operator $B$ on both sides of Eq.(\ref{eq181}), we obtain
\begin{equation}\label{eq5}
I_{t_{0^{+}}}^{1-\gamma} u(t_0) = B u_0 - B \sum_{k=1}^p c_k I_{t_{0^{+}}}^{1-\gamma} \int_{t_0^{+}}^{t_k} K_{\alpha} (t_k-s) f(s,u(s),u(b_1(s)),\cdots, u(b_r(s)))ds.
\end{equation}

Then, we have by means of the Eq.(\ref{eq3}) and Eq.(\ref{eq5}) imply that $u$ satisfies Eq.(\ref{mild}), which complete the proof.
\end{proof}

To finish the section, we will investigate the uniqueness of classical solutions of the problem Eq.(\ref{eq.4}) satisfying the condition Eq.(\ref{eq.41}), through the following result.

\begin{theorem}\label{teo65}
Suppose that
\begin{enumerate}
    \item $f: J \times \Omega^{r+1} \to \Omega$, $b_i: J \to J (i=1,2,\cdots,r)$ 
are continuous on $J$ and there is $C > 0$ such that
\begin{equation}
\left\Vert f(s,z_{0},z_{1},\cdots ,z_{r})-f(s,\overline{z}_{0},\overline{z}_{1},\cdots ,\overline{z}_{r}\right\Vert \leq C\left(\left\vert s-\overline{s}\right\vert +\sum_{i=0}^{r}\left\Vert z_{i}-\overline{z}_{i}\right\Vert _{C_{1-\gamma }}\right) ;
\end{equation}
for $s,\overline{s}  \in J$, $z_i,\overline{z}_i \in \Omega$, $i=0,1,\cdots,r$;

    \item $(r+1)MCa^{2} \left(1 + M \left\Vert B \right\Vert \sum_{k=1}^p |c_k| \right) < 1$;
    
    \item $u_0 \in \Omega$.
    
    \item Then the fractional functional differential nonlocal evolution problem {\rm Eq.(\ref{eq.4})} has a unique mild solution denoted by $u$. Moreover, if $Bu_0 \in D(\mathcal{A})$ and  
\begin{equation}\label{eqiv}
B \int_{t_0^{+}}^{t_k} K_{\alpha} (t_k-s) f(s,u(s),u(b_1(s)),\cdots, u(b_r(s)))\, {\rm{d}}s \in  D(\mathcal{A})
\end{equation}
$k=1,2,\cdots,p$ and if there is $K > 0$ such that
\begin{equation}\label{eq222}
\left\Vert u(b_{i}(s))-u(b_{i}(\overline{s}))\right\Vert _{C_{1-\gamma
}}\leq \left\Vert u(s)-u(\overline{s})\right\Vert ,\text{ for }s,\overline{s}%
\in J
\end{equation}
then $u$ is the unique classical solution to the problem {\rm Eq.(\ref{eq.4})}.
\end{enumerate}
\end{theorem}

\begin{proof}

Note that, Eq.(\ref{eq.4}) satisfying Eq.(\ref{eq.41}), admits a unique solution, since the conditions of Theorem \ref{teo2} are satisfied. On the other hand, it remains to be proved that $u$ is indeed a unique classical solution for Eq.(\ref{eq.4}). Consider
\begin{equation}\label{eq32}
N:=\underset{s\in J}{\max }\left\Vert f(s,u(s),u(b_{1}(s)),\cdots
,u(b_{r}(s)))\right\Vert  
\end{equation}
and note that
\begin{eqnarray}\label{eq221}
&&u(t+h)-u(t)  \nonumber \\
&=&\left( \mathbb{S}_{\alpha ,\beta }(t+h-t_{0})-\mathbb{S}(t-t_{0})\right)
Bu_{0}-\left( \mathbb{S}_{\alpha ,\beta }(t+h-t_{0})-\mathbb{S}%
(t-t_{0})\right) B\times   \nonumber \\
&&\times \sum_{k=1}^{p}C_{k}I_{t_{0^{+}}}^{1-\gamma
}\int_{t_{0}^{+}}^{t_{k}}K_{\alpha }(t_{k}-s)f(s,u(s),u(b_{1}(s)),\cdots
,u(b_{r}(s)))ds  \nonumber \\
&&+\int_{t_{0}^{+}}^{t_{0}+h}K_{\alpha }(t+h-s)f(s,u(s),u(b_{1}(s)),\cdots
,u(b_{r}(s)))ds  \nonumber \\
&&+\int_{t_{0}^{+}}^{t}K_{\alpha }(t-s)\left[ 
\begin{array}{c}
f(s,u(s+h),u(b_{1}(s+h)),\cdots ,u(b_{r}(s+h))) \\ 
-f(s,u(s),u(b_{1}(s)),\cdots ,u(b_{r}(s)))%
\end{array}%
\right] ds
\end{eqnarray}
for $t\in \lbrack t_{0},t_{0}+a]$, $h>0$ and $t+h\in (t_{0},t_{0}+a)$.

Consequently by Eq.(\ref{eq222}), Eq.(\ref{eq32}), Eq.(\ref{eq221}) and a condition 4, we obtain
\begin{eqnarray*}
&&\left\Vert u(t+h)-u(t)\right\Vert  \nonumber   \\
&=&\left\Vert 
\begin{array}{c}
\left( \mathbb{S}_{\alpha ,\beta }(t+h-t_{0})-\mathbb{S}_{\alpha ,\beta
}(t-t_{0})\right) Bu_{0}-(\mathbb{S}_{\alpha ,\beta }(t+h-t_{0})-\mathbb{S}_{\alpha ,\beta }(t-t_{0}))B\times  \\ 
\times \sum_{k=1}^{p}C_{k}I_{t_{0^{+}}}^{1-\gamma }\displaystyle\int_{t_{0}^{+}}^{t_{k}}K_{\alpha }(t_{k}-s)f(s,u(s),u(b_{1}(s)),\cdots ,u(b_{r}(s)))ds \\ 
+\displaystyle\int_{t_{0}^{+}}^{t_{0}+h}K_{\alpha }(t+h-s)f(s,u(s),u(b_{1}(s)),\cdots ,u(b_{r}(s)))ds \\ 
+\displaystyle\int_{t_{0}^{+}}^{t}K_{\alpha }(t-s)[f(s,u(s+h),u(b_{1}(s+h)),\cdots ,u(b_{r}(s+h))) \\ 
-f(s,u(s),u(b_{1}(s)),\cdots ,u(b_{r}(s)))]ds 
\end{array}%
\right\Vert  
\end{eqnarray*}

\begin{eqnarray}\label{eq34}
&\leq &\left( \left\Vert \mathbb{S}_{\alpha ,\beta }(t+h-t_{0})\right\Vert +\left\Vert \mathbb{S}_{\alpha ,\beta }(t-t_{0})\right\Vert \right) \left\Vert Bu_{0}\right\Vert +\left( \left\Vert \mathbb{S}_{\alpha ,\beta }(t+h-t_{0})\right\Vert +\left\Vert \mathbb{S}_{\alpha ,\beta }(t-t_{0})\right\Vert
\right) \left\Vert B\right\Vert \times  \nonumber \\
&&\times \sum_{k=1}^{p}C_{k}\left\Vert I_{t_{0^{+}}}^{1-\gamma }\right\Vert\int_{t_{0}^{+}}^{t_{k}}\left\Vert K_{\alpha
}(t_{k}-s)\right\Vert \left\Vert f(s,u(s),u(b_{1}(s)),\cdots
,u(b_{r}(s)))\right\Vert ds  \nonumber \\
&&+\int_{t_{0}^{+}}^{t_{0}+h}\left\Vert K_{\alpha }(t+h-s)\right\Vert
\left\Vert f(s,u(s),u(b_{1}(s)),\cdots
,u(b_{r}(s)))\right\Vert ds  \nonumber \\
&&+\int_{t_{0}^{+}}^{t}\left\Vert K_{\alpha }(t-s)\right\Vert \left\Vert 
\begin{array}{c}
f(s,u(s+h),u(b_{1}(s+h)),\cdots ,u(b_{r}(s+h))) \\ 
-f(s,u(s),u(b_{1}(s)),\cdots ,u(b_{r}(s)))%
\end{array}%
\right\Vert ds  \nonumber \\
&\leq &2Mh\left\Vert Bu_{0}\right\Vert +2M\widetilde{C}%
h\left\Vert B\right\Vert \sum_{k=1}^{p}C_{k}\times   \nonumber
\\
&&\times \int_{t_{0}^{+}}^{t_{k}}\left\Vert K_{\alpha
}(t_{k}-s)f(s,u(s),u(b_{1}(s)),\cdots ,u(b_{r}(s)))\right\Vert _{C_{1-\gamma
}}ds  \nonumber \\
&&+MNh+MC\int_{t_{0}^{+}}^{t}\left\Vert u(s+h)-u(s)\right\Vert _{C_{1-\gamma
}}ds+MCk\int_{t_{0}^{+}}^{t}\left\Vert u(s+h)-u(s)\right\Vert
ds  \nonumber \\
&&+\cdots +MCkr\int_{t_{0}^{+}}^{t}\left\Vert u(s+h)-u(s)\right\Vert
ds  \nonumber \\
&=&2Mh\left\Vert Bu_{0}\right\Vert +2M\widetilde{C}%
h\left\Vert B\right\Vert \sum_{k=1}^{p}C_{k}\times  \nonumber \\
&&\times\int_{t_{0}^{+}}^{t_{k}}\left\Vert K_{\alpha
}(t_{k}-s)f(s,u(s),u(b_{1}(s)),\cdots ,u(b_{r}(s)))\right\Vert ds  \nonumber \\
&&+MNh+MC(1+k+kr)\int_{t_{0}^{+}}^{t}\left\Vert u(s+h)-u(s)\right\Vert
ds  \nonumber \\
&=&\left[ 
\begin{array}{c}
2M\left\Vert Bu_{0}\right\Vert +2M\widetilde{C}\left\Vert
B\right\Vert \sum_{k=1}^{p}C_{k}\times  \\ 
\times \displaystyle\int_{t_{0}^{+}}^{t_{k}}\left\Vert K_{\alpha
}(t_{k}-s)f(s,u(s),u(b_{1}(s)),\cdots ,u(b_{r}(s)))\right\Vert ds+MN%
\end{array}%
\right] h  \nonumber \\
&&+MC(1+k+kr)\int_{t_{0}^{+}}^{t}\left\Vert u(s+h)-u(s)\right\Vert
ds  \nonumber \\
&=&\widetilde{\delta }h+MC(1+k+kr)\int_{t_{0}^{+}}^{t}\left\Vert
u(s+h)-u(s)\right\Vert ds
\end{eqnarray}
for $t\in \lbrack t_{0},t_{0}+a)$, $h>0$ and $t+h\in (t_{0},t_{0}+a)$, where 
\begin{eqnarray*}
&&\widetilde{\delta }:=2M\left\Vert Bu_{0}\right\Vert +MN2M
\widetilde{C}\left\Vert B\right\Vert \sum_{k=1}^{p}C_{k}\times  \nonumber \\
&&\times\int_{t_{0}^{+}}^{t_{k}}\left\Vert K_{\alpha
}(t_{k}-s)f(s,u(s),u(b_{1}(s)),\cdots ,u(b_{r}(s)))\right\Vert ds.
\end{eqnarray*}

Note that, for $0<\alpha <1$, $t\in \lbrack t_{0},t_{0}+a)$, Eq.(
\ref{eq34}) can rewritten in the following form 
\begin{eqnarray}\label{eq34*}
\left\Vert u(t+h)-u(t)\right\Vert &\leq& \widetilde{\delta }
h+MC(1+k+rk)\int_{t_{0}^{+}}^{t}\left\Vert u(s+h)-u(s)\right\Vert ds  \\
&\leq &\widetilde{\delta }h+MC(1+rk)t^{1-\alpha
}\int_{t_{0}^{+}}^{t}(t-s)^{\alpha -1} \left\Vert u(s+h)-u(s) \right\Vert ds 
\nonumber \\
&\leq &\widetilde{\delta }h+MC(1+rk)p\int_{t_{0}^{+}}^{t}(t-s)^{\alpha
-1}\left\Vert u(s+h)-u(s)\right\Vert ds.\nonumber
\end{eqnarray}%

By means of inequality (\ref{eq34*}) and of Gronwall inequality (see Lemma \ref{gronlema} ) we get 
\begin{eqnarray}
\left\Vert u(t+h)-u(t)\right\Vert _{C_{1-\gamma }} &\leq &\widetilde{\delta }%
h\mathbb{E}_{\alpha }\left[ MC(1+rk)\Gamma (\alpha )(t-s)^{\alpha }\right]  \nonumber \\
&\leq &\widetilde{\delta }h\mathbb{E}_{\alpha }\left[ MC(1+rk)\Gamma (\alpha )a^{\alpha }\right] 
\end{eqnarray}
for $t\in \lbrack t_{0},t_{0}+a)$, $h>0$, $t+h\in (t+0,t_{0}+a)$ and $\mathbb{E}_{\alpha }(\cdot )$ is the classical Mittag-Leffler function of a parameter.

Note that the continuity of $u$ and of $f$ on $J$ and $J\times\Omega^{k+1}$, respectively, implies that $t\rightarrow f(t,u(t),u(b_{1}(t)),\cdots ,u(b_{r}(t)))$ is Lipschitz continuous in the interval $J$. In this sense, we have to from this fact and the conditions of Theorem \ref{teo2} imply that, by Theorem \ref{teo1} the linear fractional Cauchy problem
\begin{equation}\label{eq35}	
{}^H\mathbb{D}_{t_{0}^{+}}^{\alpha ,\beta }v(t)+\mathcal{A}v(t)=f\left( t,u(t),u(b_{1}(t)),\cdots ,u(b_{r}(t))\right) ,\,\,t\in J/\{t_{0}\}  
\end{equation}%
and 
\begin{equation}\label{eq36}
I_{t_{0^{+}}}^{1-\gamma}v(t_{0})=u_{0}-\sum_{k=1}^{p}C_{k}I_{t_{0^{+}}}^{1-\gamma }u(t_{k})
\end{equation}%
has a unique solution $v$ such that 
\begin{equation}\label{eq37}
v(t)=\mathbb{S}_{\alpha ,\beta }(t-t_{0})I_{t_{0^{+}}}^{1-\gamma }v(t_{0})+\int_{t_{0}^{+}}^{t}K_{\alpha }(t-s)f(s,u(s),u(b_{1}(s)),\cdots ,u(b_{r}(s)))ds  
\end{equation}
with $t\in J$. Now, we shall show that 
\begin{equation}\label{eq38}
u(t)=v(t),\text{ }t\in J.  
\end{equation}

Note that by means of the Eq.(\ref{eq36}), remark \ref{rem1} and Eq.(\ref{mild}), we have 
\begin{eqnarray}
&&I_{t_{0^{+}}}^{1-\gamma }v(t_{0})=I_{t_{0^{+}}}^{1-\gamma
}u(t_{0})=Bu_{0}-B\sum_{k=1}^{p}C_{k}I_{t_{0^{+}}}^{1-\gamma
}\times  \nonumber \\
&&\times\int_{t_{0}^{+}}^{t_{k}}K_{\alpha }(t_{k}-s)f(s,u(s),u(b_{1}(s)),\cdots
,u(b_{r}(s)))ds.
\end{eqnarray}

Soon,
\begin{eqnarray}\label{eq39}
\mathbb{S}_{\alpha ,\beta }(t-t_{0})I_{t_{0^{+}}}^{1-\gamma }v(t_{0})&=&%
\mathbb{S}_{\alpha ,\beta }(t-t_{0})Bu_{0}-\mathbb{S}_{\alpha ,\beta
}(t-t_{0})B\times  \\
&&\times \sum_{k=1}^{p}C_{k}I_{t_{0^{+}}}^{1-\gamma
}\int_{t_{0}^{+}}^{t_{k}}K_{\alpha }(t_{k}-s)f(s,u(s),u(b_{1}(s)),\cdots
,u(b_{r}(s)))ds  \nonumber
\end{eqnarray}
$t\in J$. 

Next, from \textnormal{Eq.(\ref{eq37})}, \textnormal{Eq.(\ref{eq39})} and Eq.(\ref{mild}), we obtain
\begin{eqnarray}
&&v(t) =\mathbb{S}_{\alpha ,\beta }(t-t_{0})Bu_{0}-\mathbb{S}_{\alpha ,\beta }(t-t_{0})B\times   \nonumber
\\
&&\times \sum_{k=1}^{p}C_{k}I_{t_{0^{+}}}^{1-\gamma
}\int_{t_{0}^{+}}^{t_{k}}K_{\alpha }(t_{k}-s)f(s,u(s),u(b_{1}(s)),\cdots
,u(b_{r}(s)))ds  \nonumber \\
&&+\int_{t_{0}^{+}}^{t}K_{\alpha }(t-s)f(s,u(s),u(b_{1}(s)),\cdots
,u(b_{r}(s)))ds  \nonumber \\
&=&u\left( t\right) 
\end{eqnarray}
$t\in J$ and, therefore, Eq.(\ref{eq38}) holds.

Thus, we conclude that $u$ is the classical solution of Eq.(\ref{eq.4})-Eq.(\ref{eq.41}). Now suppose $u^{*}\neq u$ another classical solution of Eq.(\ref{eq.4})-Eq.(\ref{eq.41}) in the interval $J$. Through the Theorem \ref{teo31}, we have that $u^{*}$ is a solution to the problem Eq.(\ref{eq.4})-Eq.(\ref{eq.41}). On the other hand, the Theorem \ref{teo2} guarantees the uniqueness of a solution for Eq.(\ref{eq.4})-Eq.(\ref{eq.41}), so we conclude that $u=u^{*}$. Therefore, we conclude the result.
\end{proof}

\section{Concluding remarks}
In this paper, we investigate the existence and uniqueness of mild and classical solutions to the fractional functional differential evolution equation through the Cauchy contraction theorem and Gronwall inequality. We were able to present new results that actually contribute to the advancement and growth of the area. However, there are still some issues that deserve special mention, namely, investigating mild, strong, classical solutions to fractional differential equations introduced by means of the $\psi$-Hilfer fractional derivative \cite{ze,ze5}. But for such a success, it is necessary and sufficient condition to obtain an integral transform to investigate such solutions. Recently Jarad and Abdeljawad \cite{Jarad,15} introduced the Laplace transform with respect to another function, and the inverse version of the Laplace transform with respect to another function has yet to be presented. Studies in this sense have been investigated and certainly contributed greatly to a new path of problems to be investigated.

\section*{Acknowledgment}

JVCS acknowledges the financial support of a PNPD-CAPES (process number nº88882.305834/2018-01) scholarship of the Postgraduate Program in Applied Mathematics of IMECC-Unicamp.


\end{document}